\documentclass[11pt,a4paper]{amsart}
\usepackage{amsmath} 
\usepackage{lmodern}
\usepackage{amsmath, amssymb}
\usepackage[all]{xy}
\usepackage[utf8]{inputenc}
\usepackage[english]{babel}
\usepackage{graphicx}
\usepackage{hyperref}
\usepackage{fancyhdr}
\usepackage{amsthm}

\newtheorem{thm}{Theorem}[section]
\newtheorem{thmx}{Theorem}

\newtheorem{pro}{Proposition}[section]

\newtheorem{deff}{Definition}[section]

\newtheorem{lem}{Lemma}[section]

\begin{document}
\title{On the distribution of the periods of convex representations I}
\author{Abdelhamid Amroun}
\address{Laboratoire de Math\'ematiques d'Orsay, Univ. Paris-Sud, CNRS UMR 8628, 
Universit\'e Paris-Saclay, 91405 Orsay, France}
\email{abdelhamid.amroun@universite-paris-saclay.fr}

\begin{abstract} 
We prove a central limit theorem for a class of Hölder continuous cocycles with an application to stricly convex and irreducible rational representations of hyperbolic groups, introduced by Sambarino [Quantitative properties of convexe representations. Comment. Math.
Helv 89 (2014), 443-488].
\end{abstract}
\maketitle
\tableofcontents
\section{Introduction and main results}
In the paper \cite{sam1} various results on the asymptotic growth of orbital functions of $\rho(\Gamma)$ have been obtained by A. Sambarino, for a class of rationnal representations $\rho$ of hyperbolic groups $\Gamma$. The representations also include Hitchin representations of surface groups. 

In this paper, we propose a central limit type theorem for irreducible rational and strictly convex representations of hyperbolic groups. 
Let $\Gamma$ be a torsion-free discrete and cocompact subgroup of the isometry group $Isom(\tilde{M})$ of a simply connected complete manifold $\tilde{M}$ of sectional curvature $\leq -1$. The group $\Gamma$ is hyperbolic and it's boundary $\partial \Gamma$ is naturally identifiefd with the geometric boundary $\partial \tilde{M}$ of $\tilde{M}$ \cite{Ghys}. 

Let $\rho : \Gamma\longrightarrow PGL(d,\mathbb{R})$ be a strictly convex and irreducible rational representations  of $\Gamma$ (see Definition \ref{def1}).
For $g\in PGL(d,\mathbb{R})$, we define $\lambda_1 (g)$ to be the logarithm of the spectral radius of any lift of $g$ to $GL(d,\mathbb{R})$ with determinant $\pm 1$.
By the results in \cite{sam1}, we know that the numbers  $\lambda_1 (\rho(\gamma)), \gamma \in \Gamma$, are the periods of a well defined Hölder continuous cocycle $\beta :\Gamma \times \partial \Gamma \longrightarrow \mathbb{R}$ (see Sect. 2). Furthermore,  it was proved in \cite{sam1} that $\lambda_1 (\rho(\gamma))>0$ for all $\gamma \in \Gamma$, and the exponential growth rate $h_\beta$ (see  Sect. 3, Proposition \ref{P4})  of the cocycle $\beta$ is finite. We represent any conjugacy class $[\gamma]$ by a primitive $\gamma$, i.e. which can not be written as a positive power of another element of $\Gamma$. Note that, the periods $\lambda_1 (\rho(\gamma))$ depend only on the conjugacy class $[\gamma]$.
We introduce the set
\[
[\Gamma]_t=\{[\gamma]\in [\Gamma]: \lambda_1 (\rho(\gamma))\leq t\}.
\]
We have the following main results for the periods of a  strictly convex irreducible representation $\rho : \Gamma\longrightarrow PGL(d,\mathbb{R})$. 
\begin{thmx}\label{T2}
Let $\rho : \Gamma\longrightarrow PGL(d,\mathbb{R})$ be a strictly convex irreducible representation. 
There exist $h>0$, $L$ and $\sigma >0$ such that, for all $a,b \in \mathbb{R}$,
\[
hte^{-ht}\#\{[\gamma]\in [\Gamma]_t: \frac{\lambda_1 (\rho(\gamma)) -Lt}{\sigma \sqrt{t}}\in [a,b]\}
\rightarrow \mathcal{N}(a,b),
\]
as $t\rightarrow \infty$, and where we have set $\mathcal{N}(a,b) :=\frac{1}{\sqrt{2\pi}} 
\int_{a}^{b} e^{-\frac{x^2}{2}}dx$.
\end{thmx}
Let us explain the constants $h$, $L$ and $\sigma$. 
As said earlier, the numbers $\lambda_1 (\rho(\gamma))$ are the periods of a Hölder continuous cocycle $c : \Gamma \times \partial \Gamma \longrightarrow \mathbb{R}$. By the correspondance Theorem \ref{L} of Ledrappier, the cocycle $c$ is cohomologus to a Hölder continuous function $F$ with the same periods. Let $\mu$ be the measure of maximal entropy of the Anosov flow $\psi^t :\Gamma\backslash \partial \Gamma^2 \times \mathbb{R} \circlearrowleft$ obtained as the reparametrisation, by $c$ (or similarily by $F$) of the geodesic flow on $\Gamma\backslash \partial \Gamma^2 \times \mathbb{R}$ (see Sect. 3.2). Thus, applying Theorem \ref{sc} of Cantrell-Scharp (see Sect. 3.3), we obtain that $L=\int F d\mu$ and
\[
\sigma =\lim_{t\rightarrow+\infty}\int \left(\int_0^{t}F(\psi^t (x))dt-Lt \right )^2d\mu (x).
\]
The constant $h$ is the exponential growth rate of the periods of $\rho$, by the following result.
\begin{thm}[Sambarino \cite{sam1}]\label{exg}
\[
hte^{-ht}\#[\Gamma]_t\longrightarrow 1, \ as \ t\rightarrow \infty.
\]
\end{thm}
Based on the work of Ledrappier \cite{Led}, one can also characterize the constant $h$ as the topological entropy of the Anosov flow $\psi^t$ (see \cite{sam1}).

The function norm $\gamma \rightarrow \|\rho (\gamma)\|$ satisfies the following central limit theorem with the same constants $h>0$, $L$ and $\sigma$ in Theorem \ref{T2}.
\begin{thmx}\label{Cor1}
Let $\rho : \Gamma\longrightarrow PGL(d,\mathbb{R})$ be a strictly convex irreducible representation. Then, for all $a,b \in \mathbb{R}$,
\[
hte^{-ht}\#\{[\gamma]\in [\Gamma]_t: \frac{\ln \|\rho(\gamma)\| -Lt}{\sigma \sqrt{t}}\in [a,b]\}\rightarrow \mathcal{N}(a,b).
\]
as $t\rightarrow \infty$.
\end{thmx}
Consider the Cartan decomposition $Ke^{\mathfrak{a}}K$, where $\mathfrak{a}$ is the Cartant algebra of $PGL(d,\mathbb{R})$ and $K=PO(d)$, the projectivised orthogonal group of $\mathbb{R}^d$. To each element $g\in PGL(d,\mathbb{R})$ corresponds a unique $a(g)\in \mathfrak{a}$ such that $g\in Ke^{a(g)}K$. This defines the so called Cartant projection map $g\in PGL(d,\mathbb{R}) \longrightarrow a(g)\in \mathfrak{a}$. The Jordan projection $\lambda :PGL(d,\mathbb{R}) \longrightarrow \mathfrak{a}$ is defined by
$\lambda (g)=(|\lambda_1(g)|, \cdots, |\lambda_d(g)|)$, where $|\lambda_1(g)|\geq \cdots \geq |\lambda_d(g)|$ are the modulus of the eigenvalues of any lift of $g$ to $GL(d,\mathbb{R})$ with determinant $\pm 1$. 

An irreducible representation $\rho : \Gamma \longrightarrow PGL(d,\mathbb{R})$ is said Zariski-dense if $\rho (\Gamma)$ is a Zariski-dense subgroup of $PGL(d,\mathbb{R})$. The limit cone $\mathcal{L}_G$ of a Zariski-dense subgroup $G$ of $PGL(d,\mathbb{R})$ is defined by Benoist \cite{Be} as the closed cone containing $\{\lambda (g): g\in G\}$. The limit cone $\mathcal{L}_{\rho}$ of $\rho$ is by definition the limit cone of $\rho (\Gamma)$. When the irreducible representation $\rho$ is Zariski-dense, then $\mathcal{L}_{\rho}$ is convex and  has a nonempty interior  $\mathring{\mathcal{L}_{\rho}}$ \cite{Be}.
Consequently the cone $\mathcal{L}_{\rho}$ of the Zariski-dense representation $\rho$ is convex and  has a nonempty interior. Consider it's dual cone defined by,
\[
\mathcal{L}_{\rho}^{*} =\{\varphi \in \mathfrak{a}^* : \varphi_{|\mathcal{L}_{\rho}} \geq 0\}.
\] 
By the results in \cite{sam1}, for any $\varphi \in \mathcal{L}_{\rho}^{*}$, the  real numbers $\varphi( \lambda((\rho(\gamma))))$ are the periods of a Hölder continuous cocycle $\beta_\varphi : \Gamma \times \partial \Gamma \longrightarrow \mathbb{R}$ with finite exponential growth (see Sect. 3). We introduce the set,
\[
[\Gamma]_t^{\varphi} = \{[\gamma]\in [\Gamma]: \varphi(\lambda((\rho(\gamma))) \leq t\}.
\]
We have the following results.
\begin{thmx}\label{T}
Let $\rho : \Gamma\longrightarrow PGL(d,\mathbb{R})$ be a Zariski-dense strictly convex irreducible representation and $\varphi \in \mathring{\mathcal{L}_{\rho}^{*}} $. There exist $h_\varphi >0$, $L$ and $\sigma >0$ such that,
\[
h_{\varphi}te^{-h_{\varphi}t} 
\#  \{ [\gamma]\in [\Gamma]_t^{\varphi} : \frac{\varphi( \lambda((\rho(\gamma))))-Lt}{\sigma \sqrt{n}}\in [a,b]\}
\rightarrow \mathcal{N}(a,b), 
\] as $t\rightarrow \infty$.
\end{thmx}
\begin{thmx}\label{Cor2}
Let $\rho : \Gamma\longrightarrow PGL(d,\mathbb{R})$ be a Zariski-dense strictly convex irreducible representation and $\varphi \in \mathring{\mathcal{L}_{\rho}^{*}} $. Then,
\[
h_{\varphi}te^{-h_{\varphi}t} 
\#  \{ [\gamma]\in [\Gamma]_t^{\varphi} : \frac{\varphi( a((\rho(\gamma))))-Lt}{\sigma \sqrt{n}}\in [a,b]\}
\rightarrow
\mathcal{N}(a,b),
\] as $t\rightarrow \infty$, and where $h_\varphi >0$, $L$ and $\sigma$ are the constants of Theorem \ref{T}.
\end{thmx}
The constant $h_\varphi$ is the exponential growth rate of the periods of $\rho$, by the following result.
\begin{thm}[Sambarino \cite{sam1}]\label{exgg}
\[
h_{\varphi}te^{-h_{\varphi}t}\#[\Gamma]^{\varphi}_{t}\longrightarrow 1, \ as \ t\rightarrow \infty.
\]
\end{thm}
Leaving the details apart, we explain briefly the general ideas of the method.
In order to prove Theorems \ref{T2} and \ref{T}, we first begin by proving a central limit theorem (Theorem \ref{T1}) for  Hölder continuous cocycle $c: \Gamma \times \partial \Gamma \rightarrow \mathbb{R}$ with positive periods and finite exponential growth rate $h_c$. Under these assumptions, the periods of the cocycle $c$ are the periods of the periodic orbits of a well defined translation flow $\psi^{t}: \Gamma\backslash \partial^2 \Gamma \times \mathbb{R}\circlearrowleft$ (Definition \ref{D1}).
The flow $\psi^{t}$ is a transitive Anosov flow, which is obtained as reparametrization of the geodesic flow $\mathcal{G}^t$ by a Hölder continuous cocycle (up to a time rescaling, the orbits of $\psi^{t}$ and $\mathcal{G}^t$ are the same). We then apply a central limit theorem for periodic orbits of an Anosov flow, proved by Cantrell-Scharp (Theorem \ref{sc}). The proof of Theorems \ref{Cor2} and \ref{Cor1} combines Theorems \ref{T2} and \ref{T} with the results of Benoist (Propositions \ref{ProBe} and \ref{ProBee}) and Sambarino (Proposition \ref{Popsam} and \ref{Popsamb}). 
\section{Strictly convex representations}
In this section we define rational representation of the group $\Gamma$.
We begin with the definition of strictly convex representations.
Let $Gr_{d-1}(\mathbb{R}^d)$ be the Grassmannian of hyperplanes of $\mathbb{R}^d$.
\begin{deff}[Sambarino \cite{sam1} \cite{sam2}]\label{def1}
An irreducible representation $\rho : \Gamma\longrightarrow PGL(d,\mathbb{R})$  of the group $\Gamma$ is strictly convex if there exists a $\rho$-equivariant Hölder continuous maps
\[
\xi :\partial \Gamma \longrightarrow \mathbb{P} (\mathbb{R}^d), \ and \ \eta :\partial \Gamma \longrightarrow Gr_{d-1}(\mathbb{R}^d),
\]
such that $\mathbb{R}^d = \xi (x) \oplus \eta (y)$ whenever $x\ne y$.
\end{deff}
By the results in \cite{sam1, sam2}, the equivariant maps $\xi$ and $\eta$ are then unique.
Examples of strictly convex irreducible representations are the Hitchin representations of surface groups.
A Hitchin representation is an element in the connected component of $Hom (\Gamma, PSL(d, \mathbb{R}))$ containing a Fuchsian representation. This means that a Hitchin representation can be continuously deformed to a Fuchsian representation. A representation $\Gamma \longrightarrow PSL(d, \mathbb{R})$ is Fuchsian if it factors as 
\[
\Gamma\hookrightarrow PSL(2, \mathbb{R}) \longrightarrow PSL(d, \mathbb{R}),
\]
where the first arrow is the canonical injection and the second is the unique (up to a conjugacy) irreducible representation. When a Fuchsian representation satisfies the condition of proximality, then it is strictly convex representation (see \cite{sam2}). Moreover, the composition of a Zariski dense Hitchin representation of $\Gamma$ followed by some irreducible represenation
\[
\Gamma \longrightarrow  PGL(d, \mathbb{R}) \longrightarrow PGL(k, \mathbb{R})
\] is stricly convex.

Let $G$ be a connected real semisimple algebraic Lie group and $P$ a minimal parabolic subgroup of $G$. The homogenuous space $\mathcal{F}:=G/ P$ is identified with the Furstenberg boundary $\partial_{\infty}X :=K/M$ of the Riemannian symetric space $X=G/K$ of $G$, where $K$ is a (maximal) isotopy subgroup of $G$ (see \cite{wa}). The space $X$ has a nonpositive curvature and $\partial_{\infty}X$ is the visual boundary at infinity of $X$.
The parabolic subgroup $P$ is the stabilizer in $G$ of some point in $\partial_{\infty} X$. The space $\mathcal{F}:=G/ P$ is then the unique open $G$-orbit of this point. 
\begin{deff}[Sambarino \cite{sam2}]\label{def2}
A representation $\rho : \Gamma\longrightarrow G$  of the group $\Gamma$ is said hyperconvex if it is irreducible and admits a Hölder continuous equivariant map $\zeta : \partial\Gamma \longrightarrow \mathcal{F}$ such that whenever $x\ne y \in \partial\Gamma$ one has that the pair $(\zeta (x), \zeta (y))$ represents an open $G$-orbit in the product space $\mathcal{F} \times \mathcal{F}$.
\end{deff}
Let $\rho : \Gamma \longrightarrow G$ is a Zariski dense hyperconvex representation and $\Lambda : G \longrightarrow PGL (d, \mathbb{R})$ is a proximal irreducible representation. Then, the composition $\Lambda \circ \rho : \Gamma \longrightarrow PGL (d, \mathbb{R})$ is a stricly convex representation  (\cite{sam2} Lemma 7.1). One can then apply the main results for the representation $\Lambda \circ \rho$.

Flags in general position represent a particular interesting examples of open $G$-orbits. For example, when $G=PGL(d, \mathbb{R})$, $\mathcal{F}$ is the space of complete flags of $\mathbb{R}^d$ with stabilizer, the projectivised of the group $M$ of diagonal matrices with $\pm 1$ on the diagonal. Recall that a complete flag is given by an increasing sequence of subspaces,
\[
\{0\}=V_0 \subset V_1 \subset \cdots \subset V_d=\mathbb{R}^d,
\]where $dim V_i =i$ for all $i$.
We say that two flags $(V_i)_{i=1}^{d}$ and $(W_i)_{i=1}^{d}$ are in general position if, for all $i=1, \cdots, d$, we have $V_i \cap W_{d-i}=\{0\}$. Then, the set of flags in general position is precisely a $G$-open orbit in the product space $\mathcal{F} \times \mathcal{F}$.

Suppose that $\Gamma$ is the fundamental group of a  compact hyperbolic surface $\mathbb{H}^2 /\Gamma$. The boundary of the surface, and hence of $\Gamma$, is identified with $PGL(2, \mathbb{R})/P=\mathbb{P}(\mathbb{R}^2)$, where $P$ be the parabolic subgroup of $PGL(2, \mathbb{R})$ of upper-triangular matrices. Consider a Fuchsiann representation $\Gamma \hookrightarrow PSL(2, \mathbb{R}) \longrightarrow PSL(d, \mathbb{R})$. Then by (\cite{sam3} Corollary 3.3) there exists a Hölder continuous equivariant map $\zeta : \partial \Gamma \longrightarrow \mathcal{F}$ (the space of complete flags in $\mathbb{R}^d$) such that $\zeta (x)$ and $\zeta (y)$ are in general position, for all $x\ne y$. More generally, by a result of Labourie \cite{lab}, this also holds for Hitchin representations $\rho : \Gamma \longrightarrow PSL(d, \mathbb{R})$ (i.e. $\rho$ is in the same connected component of the Fuchsiann representation $\Gamma \hookrightarrow PSL(2, \mathbb{R}) \longrightarrow PSL(d, \mathbb{R})$).

Fix a strictly convex representation $\rho : \Gamma\longrightarrow PGL(d,\mathbb{R})$ and the corresponding $\rho$-equivariant map $\xi$ ( Definition \ref{def1}). There is a dynamical description of the periods of $\rho$, introduced by Sambarino \cite{sam1}, in terms of periodic orbits of some special flow (see Proposition \ref{P4}). This flow is obtained as a reparametrization of the geodesic flow by means of a Hölder continuous cocycle $\beta : \Gamma \times \partial \Gamma \longrightarrow \mathbb{R}$ defined as follows:
\[
\beta (\gamma, x) =\log \frac{\|\rho (\gamma)v\|}{\|v\|},
\]
where, by abuse of notation, we also denote by $\rho (\gamma)$ it's lift to $GL(d,\mathbb{R})$ with determinant $\pm 1$, and $v=v(x)$ is any vector of the projective line $\xi(x)$ (Definition \ref{def1}). 

Recall that a linear map of a finite dimensional vector space is proximal, if it has a unique complex eigenvalue of maximal modulus and its generalized eigenspace is one dimensional.

\begin{pro}[Sambarino \cite{sam1}]\label{P4}
Let $\rho : \Gamma\longrightarrow PGL(d,\mathbb{R})$ be a strictly convex representation. We have,
\begin{enumerate}
\item For all $\gamma \in \Gamma$, $\rho (\gamma)$ is proximal and $\xi (\gamma_+)$ is it's attractive projective line.
\item The periods of $\beta$ are positives and given by $\beta (\gamma, \gamma_{+})=\lambda_1 (\rho(\gamma))$, where $\exp \lambda_1 (\gamma)$ is the unique maximal eigenvalue of (a lift) $\rho (\gamma)$.
\item The exponential growth rate of $\beta$ is finite,
\[
h_\beta := \limsup_{t \longrightarrow \infty}
\frac{ \log \# \{ [\gamma]\in [\Gamma]: \beta(\gamma, \gamma_{+}) \leq t \} }{t} 
<\infty.
\]
\end{enumerate}
\end{pro}
On the other hand, there is a way to distinguish the $\gamma$'s for which $\rho (\gamma)$ is not proximal. This is done by the so called $(r, \epsilon)$-proximality condition introduced by Benoist \cite{Be}.
For this, recall that the Gromov product is the map $[\cdot, \cdot] : \partial \Gamma \times \partial \Gamma \longrightarrow \mathbb{R}$ defined by: 
\[
[x,y]=\log\frac{|\theta(v)|}{\|\theta\|\|v\|},
\]
where $\theta \in \eta (y)$, $v\in\xi(x)$. 
We set in what follows,
\[
\mathcal{G}(\theta, v)=\log\frac{|\theta(v)|}{\|\theta\|\|v\|},
\]
for all $(\theta, v) \in\mathbb{P}(\mathbb{R}^{d*})\times \mathbb{P}(\mathbb{R}^{d})- \{(\theta, v) : \theta (v) =0\}$.
\begin{deff}[Benoist \cite{Be}]\label{DeBen}
A linear transformation $g\in PGL(d, \mathbb{R})$  is $(r, \epsilon)$-proximal for some $r>0$ and $\epsilon >0$ if it is proximal, $\exp \mathcal{G}(g_{-}, g_{+}) >r$ and the complement of an $\epsilon$-neighborhood of $g_{-}$ is sent by $g$ to an $\epsilon$-neighborhood of $g_{+}$.  
\end{deff}
In the definition, $g_{+}$ is the attractive line (by proximality) for $g$, corresponding to the eigenvalue $\exp \lambda_1 (g)$, and $g_{-}$ the repelling hyperplane. Further, if we set $g_{+}=vect (v)$ and $g_{-}= Ker (\theta)$ then
\[
\mathcal{G}(g_{-}, g_{+})=\log\frac{|\theta(v)|}{\|\theta\|\|v\|}.
\]
The following results are two important consequences of this definition.
\begin{pro}[Benoist \cite{Be}]\label{ProBe}
Let $r>0$ and $\delta >0$. Then there exists $\epsilon >0$ such that for every $(r, \epsilon)$-proximal transformation $g$ one has
\[
\left | \log \|g\|-\lambda_1 (g) +  \mathcal{G}(g_{-}, g_{+})\right | <\delta.
\]
\end{pro}
\begin{pro}[Sambarino \cite{sam1}]\label{Popsam}
Let $\rho : \Gamma\longrightarrow PGL(d,\mathbb{R})$ be a strictly convex representation. Fix $r>0$ and $\epsilon >0$. Then the following set is finite:
\[
\{\gamma \in \Gamma : \exp ([\gamma_{-}, \gamma_{+}])>r \ and \ \rho (\gamma) \ is \ not \ (r, \epsilon)-proximal
\}.
\]
where, $[\gamma_{-}, \gamma_{+}]=\log\frac{|\theta(v)|}{\|\theta\|\|v\|}$ with $\xi (\gamma_+)=vect (v)$ and $\eta (\gamma_{-})=\theta$,
\end{pro}
Using these results we deduce the following lemma.
\begin{lem}\label{lem1}
Let $(r, \epsilon)$ as in Proposition \ref{Popsam}. Then,
\[
hte^{-h t}\#\{[\gamma]\in [\Gamma]_t:  \ \rho (\gamma) \ is \ (r,\epsilon)-proximal\}
 \longrightarrow 1, \ as \ t\rightarrow \infty.
\]
\end{lem}
\begin{proof}
First of all, observe that each conjugacy class $[\cdot]\in [\Gamma]$ has a representative $\gamma$ whose end fixed points $\gamma_{-}$ and $\gamma_{+}$ are far appart (with respect to the Gromov distance $d_G$ on the boundary $\partial \Gamma$) by a positive constant $\kappa$ independant from the class $c$ (since $\Gamma$ acts cocompactly i.e. with a compact fundamental domain). 
Then, since $\xi^-$ and $\xi^+$ are uniformely continuous, by the continuity of the positive function $ \exp [x, y]$ on the compact set $\{(x,y)\in\partial \Gamma \times \partial\Gamma: d_{G}(x,y)\geq \kappa\}$,  
there exists a positive constant $r>0$ such that every conjugacy class $[\cdot]\in [\Gamma]$ can be represented by some $\gamma \in \Gamma$ with $\exp [\gamma_{-}, \gamma_{+}] > r$.
Then we have for all $t>0$, 
\begin{equation}\label{E2}
\# [\Gamma]_t
=\#\{[\gamma]\in [\Gamma]_t  : \exp [\gamma_{-}, \gamma_{+}] >r\}.
\end{equation}
By Proposition \ref{Popsam}, all $\gamma$'s in $(1)$  
(except a finit number depending only on $r$ and $\epsilon$) are $(r, \epsilon)$-proximal. Thus, using Theorem \ref{exg}, we obtain
\[
hte^{-h t}\#\{[\gamma]\in [\Gamma]_t: \rho (\gamma)\ is\ (r, \epsilon)-proximal\}\rightarrow 1,
\]as $t$ goes to infinity.
This proves the lemma.
\end{proof}
We need an extension of the above results in order to prove Theorem \ref{Cor2} (see Sect 4.). For this, fix a norm $\|\ \|_{\mathfrak{a}}$ in the subalgebra $\mathfrak{a}$ of $PGL(d, \mathbb{R})$. There is a well defined function $\mathcal{G}_{\mathfrak{a}}: \partial^2\mathcal{F}\longrightarrow \mathfrak{a}$, called the Gromov product (see \cite{sam1}) such the following holds.
\begin{pro}[Benoist \cite{Be}]\label{ProBee}
Let $r>0$ and $\delta >0$. Then there exists $\epsilon >0$ such that for every $(r, \epsilon)$-proximal transformation $g$ one has
\[
\left \|a(g)-\lambda (g) +  \mathcal{G}_{\mathfrak{a}}(g_{-}, g_{+})\right \|_{\mathfrak{a}} <\delta.
\]
\end{pro}
\begin{pro}[Sambarino \cite{sam1}]\label{Popsamb}
Let $\rho : \Gamma\longrightarrow PGL(d,\mathbb{R})$ be a strictly convex representation. Fix $r>0$ and $\epsilon >0$. Then the following set is finite:
\[
\{\gamma \in \Gamma : \exp (\|\mathcal{G}_{\mathfrak{a}}(\eta (\gamma_{-}), \xi (\gamma_{+}))\|_{\mathfrak{a}})>r \ and \ \rho (\gamma) \ is \ not \ (r, \epsilon)-proximal
\}.
\]
\end{pro}
Using Proposition \ref{Popsamb} and Theorem \ref{exg} we deduce the following result (with the same proof as in Lemma \ref{lem1}).
\begin{lem}\label{lemm1}
Let $\varphi \in \mathring{\mathcal{L}_{\rho}^{*}} $ and $(r, \epsilon)$ as in Proposition \ref{Popsamb}. Then,
\[
h_{\varphi}te^{-h_\varphi t}\#\{[\gamma]\in [\Gamma]_t^\varphi:  \ \rho (\gamma) \ is \ (r,\epsilon)-proximal\}
 \longrightarrow 1, 
\]as $t\rightarrow \infty$.
\end{lem}

\section{Definitions and auxilliary results}
\subsection{Hölder continuous cocycles}
The main reference in this section is the Ledrappier's paper \cite{Led}.
\begin{deff}\label{def2}
A cocycle over $\partial \Gamma$, is a real valued function
$c:  \Gamma \times \partial \Gamma\rightarrow \mathbb{R}$,
such that, for all $\gamma_1, \gamma_2 \in \Gamma$ and $\xi \in \partial \Gamma$, 
\[
c(\gamma_1 \gamma_2, \xi)=c(\gamma_2, \gamma_1\cdot \xi)+c(\gamma_1, \xi).
\]
If, for all $\gamma \in \Gamma$, the map $\xi \rightarrow c(\gamma, \xi)$ is Hölder continuous on $\partial \Gamma$, we say that the cocycle $c$ is Hölder. The cocycle $c$ is positive if $c(\gamma, \xi)>0$, for all $(\gamma, \xi)\in   \Gamma \times \partial \Gamma$.
\end{deff}
Two Hölder cocycles are cohomologically equivalent if they differ by a Hölder continuous function $U: \partial M\rightarrow \mathbb{R}$ such that,
\[
c_1(\gamma, \xi)+c_2(\gamma, \xi)=U(\gamma \cdot \xi)-U(\xi).
\]
Given $\gamma \in \Gamma$, recall that $\gamma_+$ it's attractive fixed point in $\partial \Gamma$.
The numbers $c(\gamma, \gamma_{+})$ depend only on the conjugacy class $[\gamma]\in \Gamma$ and on the cohomological class of $c$ \cite{Led}. We call $c(\gamma, \gamma_{+})$ the periods of $c$. 
Recall the following important result by Ledrappier (see \cite{Led} p104-105).
\begin{thm}[Ledrappier \cite{Led}]\label{L}
The Liv$\check{s}$ic cohomological classes of $\Gamma$-invariant $\mathcal{C}^2$- functions $F: T^1M \rightarrow \mathbb{R}$
are in one-to-one correspondance with the cohomological classes of Hölder cocycles $c:  \Gamma \times \partial \Gamma\rightarrow \mathbb{R}$.
Moreover, the classes in correspondance have the same periods given by $c(\gamma, \gamma^{+})=\int_{p}^{\gamma p}F$, the integral of $F$ over the geodesic segment $[p, \gamma p]$ (for any $p\in M$).
\end{thm}
The periods depending only on the conjugacy classe considered, we will set $\int_{p}^{\gamma p}F=\int_{[\gamma]}F$.
\subsection{A central limit theorem for Hölder continuous cocycles}
For a cocycle $c: \Gamma \times \partial \Gamma \rightarrow \mathbb{R}$ with positive periods, we set
\[
[\Gamma]_t^c=\# \{ [\gamma]\in [\Gamma]: c(\gamma, \gamma_{+}) \leq t \}.
\]
\begin{thm}\label{T1}
Let $c: \Gamma \times \partial \Gamma \rightarrow \mathbb{R}$ be a Hölder continuous cocycle with positive periods and finite exponential growth rate $h_c$,
\[
h_c := \limsup_{t \longrightarrow \infty}
\frac{ \log \#[\Gamma]_t^c}{t} 
<\infty.
\]
There exist two constants $L_c>0$ and $\sigma_c>0$ such that for any $a, b\in \mathbb{R}$ with $a<b$ we have
\[
h_cte^{-h_c t} \# \{[\gamma]\in [\Gamma]_t^c : \frac{c(\gamma, \gamma_{+}) -L_ct}{\sigma_c \sqrt{t}}\in [a,b]\}
\rightarrow
 \mathcal{N}(a, b), \ as \ t\rightarrow \infty.
\]
\end{thm}
The constants $L_c$ and $\sigma_c$ depend only on the cocycle $c$.
By the correspondance Theorem \ref{L} of Ledrappier, the cocycle $c$ is cohomologus to a Hölder continuous function $F_c$ with the same periods. Let $\mu_c$ be the measure of maximal entropy of the Anosov flow $\psi^t :\Gamma\backslash \partial \Gamma^2 \times \mathbb{R} \circlearrowleft$ obtained as the reparametrisation by $c$ (or $F$) of the geodesic flow on $\Gamma\backslash \partial \Gamma^2 \times \mathbb{R}$ (see Theorem \ref{samb} Sect. $5.3$). Thus, applying Theorem \ref{sc} (see Sect. $5.3$), we obtain that $L_c=\int F_c d\mu_c$ and
\[
\sigma_c =\lim_{t\rightarrow+\infty}\int \left(\int_0^{t}F_c(\psi^t (x))dt-L_ct \right )^2d\mu_c (x).
\]
\subsection{Reparametrization of a Hölder continuous cocycle}
Let $X$ be a compact metric space and $\varphi^{t}: X \rightarrow X$ a continuous flow without fixed points. We consider in this paper, Hölder continuous cocycles $c$ over the flow $\varphi^{t}$.
A cocycle $c$ over the flow $\varphi^{t}$, is a function $c :X\times \mathbb{R} \rightarrow \mathbb{R}$ which satisfies the conditions:
\begin{itemize}
\item For all $x\in X$ and $s, t \in \mathbb{R}$,
\begin{equation}\label{E1}
c(x, s+t)=c(\varphi_t (x), s) + c(x,t),\ and
\end{equation}
\item For all $t\in \mathbb{R}$, the function $x\rightarrow c(x,t)$ is Hölder continuous (the exponent being independent from $t$).
\end{itemize}
Fundamental examples of such cocycles are given by a Hölder continuous function $F: X\times \mathbb{R \rightarrow} \mathbb{R}$ by setting,
$c_F (x,t)=\int_{0}^{t}F( \varphi_s (x) )ds$ for $t\geq 0$ and, 
$c_F (x,t)= -c_F (\varphi_t (x),-t)$, for $t<0$.

A cocycle $c$ is positive if for all $x\in X$, $c(x,t)>0$ for all $t\in \mathbb{R}$. In this case, for each $x\in X$, the function $t\rightarrow c(x,t)$ is increasing and in fact, it defines a homeomorphism of $\mathbb{R}$. It makes sens to consider the inverse cocyle $\hat{c}$:
\begin{equation}\label{E2}
\hat{c} (x, c(x,t))=c(x, \hat{c}(x,t))=t, \forall \ x\in X.
\end{equation}
\begin{deff}\label{D1}
The reparametrization of the flow $\varphi$ by the positive cocycle $c$, is the flow $\psi$ defined for all $x\in X$ by $\psi^{t}(x)= \varphi^{\hat{c}(x,t)}(x)$. The flow $\psi$ is indeed Hölder by (1) and (2). Furthermore, both share the same periodic orbits; if $p(\theta)$ is the period of the periodic $\varphi$-orbit $\theta$, then $c(x, p(\theta))$ is it's period as a periodic $\psi$-orbit for all $x\in \theta$. 
\end{deff}
Recall the following reparametrization theorem of Sambarino \cite{sam1}.
\begin{thm}[Sambarino \cite{sam1}]\label{samb}
Let $c$ be a Hölder cocycle with positive periods such that $h_c$ is finite and positive. Then the following holds.
\begin{enumerate}
\item The action of $\Gamma$ on $\partial^2\Gamma \times \mathbb{R}$,
\[
\gamma(x,y,s)=(\gamma x, \gamma y, s-c(\gamma, y)),
\]
is proper and cocompact. The translation flow $\psi^t: \Gamma \backslash\partial^2\Gamma \times \mathbb{R} \circlearrowleft$,
\[
\psi^t \Gamma(x,y,s)=\Gamma(x, y, s-t)),
\]
is conjugated to a Hölder reparametrization of the geodesic flow $\mathcal{G}^t :\Gamma \backslash T^1\tilde{M}\circlearrowleft$. Note that $\psi^t$ is a transitive Anosov flow.
\end{enumerate}
\end{thm}
By Ledrappier's Theorem \ref{L}, one can set $c=c_F$ for some Hölder continuous function $F$.
The reparametrization in the above theorem means that there exists a $\Gamma$-equivariant homeomorphism $E : T^1\tilde{M} \longrightarrow \partial^2\Gamma \times \mathbb{R}$ such that, for all $x=(p, v)\in T^1\tilde{M}$
\[
E(\mathcal{G}^{t} (p,v))=\psi^{c(x,t)}(E(p,v)),
\]where $c(x,t)=\int_0^t F(\mathcal{G}^s(x))ds$. It was proved in \cite{sam1} that $E$ is given by
\[
E(p,v)=(v_{-}, v_{+}, B_{v_{+}}^F)(p,o),
\]
where $o\in \tilde{M}$ is a fixed point (a base point), $v_{-}$ and $v_{+}$ are the end points in $\partial \tilde{M}$ of the geodesic with origine at $(p, v)\in T^1\tilde{M}$, and $\partial \tilde{M}\ni\xi \rightarrow B_{\xi}^F$ is the Busemann function based on the function $F$ (see \cite{sam1}).

Setting $F=1$ leads to the Hopf parametrization of the unit tangent bundle (based on the Busemann cocycle $B_{v_{+}}$),
\[
T^1\tilde{M}\ni(p,v)\longrightarrow (v_{-}, v_{+}, B_{v_{+}}(p,o)) \in \partial^2\Gamma \times \mathbb{R}.
\]
The geodesic flow of $T^1\tilde{M}$ is, by the way, a translation flow on $\partial^2\Gamma \times \mathbb{R}$,
\[
\mathcal{G}^t (v_{-}, v_{+}, B_{v_{+}}(p,o))=(v_{-}, v_{+}, B_{v_{+}}(p,o)+t).
\] 
\subsection{A central limit theorem for hyperbolic flows}
The geodesic flow of a negatively curved compact manifold does not admit a cross section. This a consequence of the Preissman theorem \cite{pr} (see \cite{sam3}).This means essentially that the flow $\phi^t$ is not a suspension of a continuous flow. Recall that a cross section for a flow $\phi^t : X\circlearrowleft$ is a closed subset $K$ of $X$ such that the function $K\times \mathbb{R} \ni (x,t)\rightarrow \phi^t (x)$ is a surjective local homeomorphism. Moreover, this property is invariant under reparametrization \cite{sam1}. More precisely, the flow $\phi^t$ admits a cross section if and only if the same is true for any reparametrization $\psi^t$ of $\phi^t$. Consequently, the flow $\psi^t$ of Theorem \ref{samb} does not admit a cross section. This is equivalent to say that the subgroup of $\mathbb{R}$ generated by the periods of $c=c_F$ (i.e. the subgroup generated by $\{\int_{\tau}F: \tau \ periodic\}$) is dense \cite{sam3}. We can thus apply the following result of Cantrell and Scharp \cite{cs} (see also \cite{cs}, Remark 6.4 ) to the periodic orbits $\tau$ of the flow $\psi^t$.
\begin{thm}[Cantrell-Scharp \cite{cs}]\label{sc}
Suppose that $\psi^t : \Lambda \longrightarrow \Lambda$ is either a transitive Anosov flow with stable and unstable foliations which are not jointly integrable or a hyperbolic flow satisfying the approximability condition. Let $f:\Lambda \rightarrow\mathbb{R}$ be a Hölder continuous function that is not a coboundary. Then, there exists two constants $L$ and $\sigma_f >0$ such that, for all $a,b\mathbb{R}$,
\[
\frac{\#\{\tau \ periodic,\ p(\tau)\leq t: 
\frac{\int_{\tau}f-Lt}{\sigma_f\sqrt{t}}\in [a, b]\}}{\#\{\tau \ periodic,\ p(\tau)\leq t\}}\longrightarrow \mathcal{N}(a, b), 
\]as $t\rightarrow \infty$.
\end{thm}
In the theorem, $p(\tau)$ is the least period of the periodic orbit $\tau$. Let $\mu$ be the measure of maximal entropy of the flow $\psi^t$. The constants $L$ and $\sigma_f$ are given respectively by $L=\int fd\mu$ and
\[
\sigma_f =\lim_{t\rightarrow+\infty}\int_{\Lambda} \left(\int_0^{t}f(\psi^t (x))dt-t\int fd\mu \right )^2d\mu (x).
\]
Note that this theorem is more general than Theorem \ref{T1}, in the sens that we have not necessarily $p(\tau)=\int_\tau f$ for the periodic orbits $\tau$ of $\psi^t$.  

\section{Proof of the main results}
\subsubsection{Proof of Theorem \ref{T1}}
\begin{proof}
Under the assumptions of Theorem \ref{T1} ($c$ Hölder continuous and $0<h_c <\infty$), one can apply the reparametrizing Theorem \ref{samb}. We have then a proper and cocompact action $\Gamma \curvearrowright \backslash \partial^2 \Gamma \times \mathbb{R}$
\[
\gamma (x,y,s)=(\gamma x, \gamma y, s-c(\gamma, y)),
\]
and a translation flow on the quotient space, $\psi^{t}: \Gamma \backslash \partial^2 \Gamma \times \mathbb{R} \circlearrowleft$
\[
\psi^{t} \Gamma(x,y,s)=\Gamma (x,y, s-t).
\]
For $\gamma \in \Gamma$ primitive, the periods $c(\gamma, \gamma^+)$ of $c$ are the periods $p(\tau)$ of the periodic orbits $\tau$ of $\psi^{t}: \Gamma \backslash \partial^2 \Gamma \times \mathbb{R} \circlearrowleft$. 

Let $s\rightarrow \tau (s)=\Gamma \cdot (\gamma_{-}, \gamma_{+},s)$ a periodic orbit of $\psi^t$, i.e. the lift of $\tau$ to $\partial^2 \Gamma \times \mathbb{R}$, is $\tau=(\gamma_{-}, \gamma_{+}, s)$, with $\gamma \in \Gamma$ primitive and $s\in \mathbb{R}$  (we denote the orbit and the lifts by the same symbol if there is no confusion to be worried about). Since $\gamma (\gamma_{-}, \gamma_{+}, s)=(\gamma_{-}, \gamma_{+}, s-c(\gamma, \gamma^+))$, we have $p(\tau)=c(\gamma, \gamma^+)$. This is a straightforward verification by observing that, for all $s\in \mathbb{R}$,
\[
\Gamma \cdot (\gamma_{-}, \gamma_{+}, s-c(\gamma, \gamma_{+}))= \Gamma \cdot \gamma (\gamma_{-}, \gamma_{+}, s)=\Gamma \cdot (\gamma_{-}, \gamma_{+}, s)=\tau (s)=\tau (s-p(\tau))
\]
\[=\Gamma \cdot (\gamma_{-}, \gamma_{+}, s-p(\tau)).
\]
The flow $\psi^{t}$ is the reparametrization of the geodesic flow by a Hölder continuous cocycle $c_F$, with $F>0$,
and we have by Theorem \ref{L},
\begin{equation*}\label{éq}
p(\tau)=c(\gamma, \gamma^+)=c_F (\gamma, \gamma^+)=\int_{\tau}F.
\end{equation*}
Then, 
\begin{eqnarray*}\label{EQ}
&&\frac{\#\{[\gamma] \in [\Gamma]_t^c: 
\frac{c(\gamma, \gamma_{+})-Lt}{\sigma\sqrt{t}}\in [a, b]\}}
{\#[\Gamma]_t^c}\\
&=&
\frac{\#\{\tau \ periodic, \ p(\tau)\leq t: 
\frac{\int_{\tau}F-Lt}{\sigma\sqrt{t}}\in [a, b]\}}{\#\{\tau \ periodic, \ p(\tau)\leq t\}}.
\end{eqnarray*}
Consequently, Theorem \ref{T1} is now a consequence of Thoerem \ref{exg} and Theorem \ref{sc}.
\end{proof}

\subsubsection{Proof of Theorem \ref{T2} and Theorem \ref{T}}
\begin{proof}
The periods of $\beta_1$ are positive, $h_{\beta_{1}} <\infty$ and $\varphi \in \mathring{\mathcal{L}_{\rho}} $
$\lambda_1 (\rho(\gamma))=\beta_1 (\gamma, \gamma^+)$. Then Theorem \ref{T2} is a consequence of Theorem \ref{T1} (with $c=\beta_1$). 
The same holds for Theorem \ref{T}. Indeed, the numbers $\varphi (\lambda (\rho\gamma))$, for $\varphi \in \mathring{\mathcal{L}_{\rho}}$, are the positive periods of a cocycle $\beta_\varphi$ with $h_{\beta_{\varphi}}<\infty$.
\end{proof}
\subsubsection{Proof of Theorem \ref{Cor1}}
\begin{proof}
Set
\[
\lambda_t(\gamma) = \frac{\lambda_1 (\rho(\gamma)) -Lt}{\sigma \sqrt{t}}, \ and  \ \delta_t (\gamma) = \frac{\ln \|\rho(\gamma)\| -\lambda_1(\rho (\gamma))}{\sigma \sqrt{t}}.
\]
Recall that by Theorem \ref{T2}, $hte^{-ht} 
\# \{ [\gamma]\in [\Gamma]_t: \lambda_t(\gamma)\in [a,b\}]$ converges to $\mathcal{N}(a, b)$.
We have to show that,
\[
hte^{-ht} 
\# \{ [\gamma]\in [\Gamma]_t: \lambda_t(\gamma)+\delta_t(\gamma)\in [a,b\}]\rightarrow \mathcal{N}(a,b),
\] as $t\rightarrow \infty$.

Choose $\delta, r$ and $\epsilon$ as in Proposition \ref{ProBe}.
By Proposition \ref{Popsam}, for all, but a finite number of the $\gamma$'s, are $(r, \epsilon)$-proximal $\gamma$ and,
\[
|\log\|\rho(\gamma)\| -\lambda_1(\rho (\gamma))+\log r|\leq 2\delta.
\]
This implies that for $t$ sufficiently large (depending on $\delta, r$ and $\sigma$) we have $|\delta_t(\gamma)|\leq \delta$, and consequently, 
\[
\# \{ [\gamma]\in [\Gamma]_t: |\delta_t(\gamma)|\leq \delta\}\geq
\# \{ [\gamma]\in [\Gamma]_t: \gamma \ is \ (r, \epsilon)-proximal\}.
\]
Thus by Lemma \ref{lem1},
\begin{equation}\label{E4}
hte^{-ht} 
\# \{ [\gamma]\in [\Gamma]_t: |\delta_t(\gamma)|\leq \delta\} \longrightarrow 1,
\ as \ t \rightarrow \infty.
\end{equation}
Now, by (\ref{E4}) we get,
\begin{eqnarray*}
&& \liminf_{t\rightarrow \infty}
hte^{-ht} 
\# \{ [\gamma]\in [\Gamma]_t: \lambda_t(\gamma)+\delta_t (\gamma)\in [a,b\}] \\
&\geq & \liminf_{t\rightarrow \infty}hte^{-ht} 
\# \{ [\gamma]\in [\Gamma]_t: \lambda_t(\gamma)
\in 
[a+\delta, b-\delta] \ and \ |\delta_t(\gamma)|
\leq \delta\}\\
&=& \liminf_{t\rightarrow \infty}hte^{-ht} 
\# \{ [\gamma]\in [\Gamma]_t: \lambda_t(\gamma)
\in [a+\delta, b-\delta]\}\\
&=& \mathcal{N}(a+\delta,b-\delta).
\end{eqnarray*}
For the $\limsup$ we have (using (\ref{E4}))
\begin{eqnarray*}
&& \limsup _{t\rightarrow \infty}hte^{-ht} 
\# \{ [\gamma]\in [\Gamma]_t: \lambda_t(\gamma)+\delta_t (\gamma)\in [a,b\}]\\
&=& \limsup _{t\rightarrow \infty}hte^{-ht} 
hte^{-ht}\# \{ [\gamma]\in [\Gamma]_t: \lambda_t(\gamma)
\in [a+\delta, b-\delta] \ and \ |\delta_t(\gamma)|
\leq \delta\}\\
&\leq &  \limsup _{t\rightarrow \infty}hte^{-ht} 
hte^{-ht}\# \{ [\gamma]\in [\Gamma]_t: \lambda_t(\gamma)
\in [a+\delta, b-\delta]\}\\
&=& \mathcal{N}(a+\delta,b-\delta).
\end{eqnarray*}
Thus, since $\delta$ is arbitrary, we finally get,
\[
\lim _{t\rightarrow \infty}hte^{-ht} 
\# \{ [\gamma]\in [\Gamma]_t: \lambda_t(\gamma)+\delta_t (\gamma)\in [a,b\}]=\mathcal{N}(a,b).
\]
This completes the proof of Theorem \ref{Cor1}.
\end{proof}
\subsubsection{Proof of Theorem \ref{Cor2}}
\begin{proof}
Fix $\delta >0$, $\epsilon >0$ and $r>0$ as in the proof of Proposition \ref{ProBee}. Then, for any $(r, \epsilon)$-proximal $\rho(\gamma)$ we have
\[
|\varphi(a(\rho(\gamma))) -\varphi(\lambda(\rho(\gamma))) +\varphi(\mathcal{G}_{\mathfrak{a}}(\eta (\gamma_{-}), \xi (\gamma_{+}))| \leq \|\varphi\|\delta.
\]
By Proposition \ref{Popsamb}, the set of $\gamma$'s with 
\[
\exp \varphi(\mathcal{G}_{\mathfrak{a}}(\eta (\gamma_{-}), \xi (\gamma_{+}))) >r
\]
and such that $\rho (\gamma)$ is not $(r, \epsilon)$-proximal is finite.
Use Lemma \ref{lemm1}, and proceed as in the above proof of Theorem \ref{Cor1} to prove Theorem \ref{Cor2}.
\end{proof}

\end{document}